\documentclass[11pt]{amsart}
\usepackage[totalwidth=440pt, totalheight=640pt]{geometry}
\usepackage{amsmath, amssymb, amsthm, amsfonts, mathrsfs, eucal, enumerate, natbib, layout}
\usepackage[normalem]{ulem}
\usepackage[all,tips,2cell]{xy}
\usepackage[usenames]{color}
\usepackage{xspace}
\usepackage{ifpdf}
\usepackage{ifthen}
\usepackage{hyperref}
\usepackage{tikz}
\usepackage{verbatim}
\usetikzlibrary{matrix,arrows}
\usepackage{mathtools}
\usepackage[usenames]{color}

\newcommand{\details}[1]{}

\newcommand{\cA}{\mathcal{A}}

\newcommand{\cP}{\mathcal{P}}
\newcommand{\cQ}{\mathcal{Q}}

\newcommand{\cX}{\mathcal{X}}

\newcommand{\PICARD}{\mathbb{P}\mathrm{icard}}

\newcommand{\STACK}{\mathbb{S}\mathrm{tack}}

\newcommand{\ass}{\mathsf{a}}

\newcommand{\id}{\mathrm{id}}
\newcommand{\comm}{\mathsf{c}}
\newcommand{\bS}{\textbf{S}}
\newcommand{\bfu}{\textbf{1}}

\newcommand{\twoX}{\mathbb{X}}
\newcommand{\twoT}{\mathbb{T}}

\newcommand{\Set}{\mathrm{Set}}
\newcommand{\Ab}{\mathrm{Ab}}
\newcommand{\Alg}{\mathrm{Alg}}
\newcommand{\T}{\mathrm{T}}
\newcommand{\C}{\mathrm{C}}

\newtheorem{theorem}{Theorem}[section]
\newtheorem*{theorem*}{Theorem}
\newtheorem{corollary}[theorem]{Corollary}
\newtheorem*{corollary*}{Corollary}

\newtheorem*{claim*}{Claim}
\newtheorem*{lemma*}{Lemma}

\newtheorem*{proposition*}{Proposition}

\newtheorem*{conjecture*}{Conjecture}
\newtheorem{def-proposition}[theorem]{Definition-Proposition}
\theoremstyle{definition}
\newtheorem{definition}[theorem]{Definition}
\newtheorem*{definition*}{Definition}

\newtheorem*{example*}{Example}
\numberwithin{equation}{section}

\begin{document}

\title[The group law of Picard stacks via matrices]
{The group law of Picard stacks via matrices}

\author{Cristiana Bertolin}
\address{Dipartimento di Matematica, Universit\`a di Padova, Via Trieste 63, Italy}
\email{cristiana.bertolin@unipd.it}

\author{Federica Galluzzi}
\address{Dipartimento di Matematica, Universit\`a di Torino, Via Carlo Alberto 10, Italy}
\email{federica.galluzzi@unito.it}



\begin{abstract}
	Let $\bS$ be a site. We show that the 2-stack of strictly commutative Picard stacks over $\bS$ is algebraic, i.e. it is 2-equivalent to the 2-stack of 2-algebras for an adequate algebraic 2-stack theory over $\bS$.
\end{abstract}

\maketitle
\tableofcontents

\section*{Introduction}

The concept of algebraic theories was originally introduced by F. W. Lawvere in his doctoral thesis \cite{La}. An algebraic theory is a category whose
objects are finite products of a given object with itself. In this framework, groups can be described as functors from an adequate algebraic theory to the category of sets, which preserve finite products (see \cite{AdRoVi}).
The aim of this paper is to generalize this description of groups to strictly commutative Picard stacks over a site $\bS.$

We start recalling
 the general properties of stacks and strictly commutative Picard stacks (see Section \ref{2-St-and-Pic-St}). We define the notions of algebraic 2-stack theory over a site $\bS$ and of 2-algebra for an algebraic 2-stack theory, which  
  generalize to 2-stacks the notions of algebraic theory and of algebra for an algebraic theory respectively (see Definition \ref{def-alg-2-st-2-alg}).
This brings us to the definition of an algebraic 2-stack, that is a 2-stack 2-equivalent to the 2-stack of 2-algebras for an adequate algebraic 2-stack theory (see Definition \ref{alg-2-st}). Finally in Theorem \ref{Teo} we prove that the 2-stack of strictly commutative Picard stacks over a site $\bS$ is algebraic.

\section{Preliminary notions}\label{2-St-and-Pic-St}

Let $\bS$ be a site. In this paper $U$ denotes an object of $\bS$.  A \textit{stack} over $\bS$ is a fibered category $\cX$ over $\bS$ such that 
\begin{itemize}
	\item (\emph{Gluing condition on objects}) descent is effective for objects in $\cX$, and
	\item (\emph{Gluing condition on arrows}) for any $U$ and for every pair of objects $X,Y$ of the category $\cX(U)$, the presheaf of arrows $\mathrm{Arr}_{\cX(U)}(X,Y)$ of $\cX(U)$ is a sheaf over $U$.
\end{itemize}

For the notions of morphism of stacks (i.e. cartesian functor), and morphism of cartesian functors, we refer to 
\cite[Chp II 1.2]{Gi}. An \textit{equivalence of stacks} $F: \cX_1 \to \cX_2$ is a morphism of stacks such that any object $Y $ of $\cX_2(U)$ is isomorphic to an object of the form $F(U)(X)$ with $X$ an object of $\cX_1(U)$, and for any pair of objects $X,Y  $ of $\cX_1(U)$, the map
$\mathrm{Arr}_{\cX_1(U)}(X,Y) \to \mathrm{Arr}_{\cX_2(U)}(F(U)(X),F(U)(Y))$ is bijective.
Two stacks are \textit{equivalent} if there exists an equivalence of stacks between them.
A \textit{stack in groupoids} over $\bS$ is a stack $\cX$ over $\bS$ such that the category $\cX(U)$ is a groupoid, i.e. a category whose arrows are invertible.

A \textit{strictly commutative Picard stack} over the site $\bS$ (just called a Picard stack) is a stack in groupoids $\cP$ over $\bS$ endowed
\begin{itemize}
	\item  with a morphism of stacks  $ +: \cP \times \cP \rightarrow \cP,$ called \emph{the group law} of $\cP$, which assigns to each object 
	$(X_1,X_2)$ of $(\cP \times_{\bS} \cP) (U)$ the object $X_1+X_2$ of $\cP(U)$,  and which assigns to each arrow $(f,g): (X_1,X_2) \to (X'_1,X'_2)$ of $(\cP \times_{\bS} \cP)(U)$ the arrow $f+g: X_1+X_2 \to X'_1+X'_2$ of $\cP (U)$, and 
	\item two 
	natural isomorphisms, called respectively the \emph{associativity} and the \emph{commutativity}
	
	$$\ass :+\circ (+\times \id_\cP) \Rightarrow +\circ (\id_\cP \times +),$$ 
	$$\comm : +\circ \mathsf{s} \Rightarrow +,$$
	with $ \mathsf{s}(X,Y)=(Y,X)$ for all $X,Y \in \cP(U)$, which express respectively the associativity and the commutativity constraints of the group law $+$ of $\cP$,
\end{itemize}
such that $\cP (U)$ is a strictly commutative Picard category (i.e. it is possible to make the sum of two objects of $\cP(U)$, this sum is associative and commutative, and any object of $\cP(U)$ has an inverse with respect to this sum, see \cite[\S 1]{B13} or \cite[\S 1.4]{De}
 for more details). The word "strictly" means that $\comm_{X,X}: X+X \to X+X$ is the identity 
for all $X \in \cP(U)$.

Any Picard stack $\cP$ admits a unique up to unique isomorphism neutral object $\mathbf{e}$, 
which can be defined as a couple $(\mathbf{e},\varphi),$ where $\mathbf{e}$ is a global object of $\cP$ and $\varphi: \mathbf{e}+\mathbf{e} \rightarrow \mathbf{e}$ is an isomorphism of $\cP$.
It exists a unique natural isomorphism
\[	l_X: \mathbf{e} + X \stackrel{\cong}{\longrightarrow} X \qquad \quad  \forall~ X \in \cP(U)\]
such that the following diagram commute
\[\xymatrix{
	\mathbf{e}+(\mathbf{e}+X)  \ar[d]_{\id_\cP+l_X} & (\mathbf{e}+\mathbf{e})+X \ar[l]_\ass \ar[d]^{\varphi +\id_\cP} \\
	\mathbf{e}+X \ar@{=}[r]& \mathbf{e}+X
}
\]
It exists also the natural isomorphism 
\[
r_X: X+\mathbf{e} \stackrel{\cong}{\longrightarrow} X \qquad  \quad \forall~ X \in \cP(U).
\]
The isomorphism $\varphi$ is a special case of these two natural isomorphisms: $\varphi=l_\mathbf{e}=r_\mathbf{e}$, and the natural isomorphism $\comm$ exchanges  $l_X$ and $r_X$. 

  Observe that if $X_i$ and $f_j$ are respectively objects and arrows of $\cP(U)$, $n_i,m_j$ are integers, and $N,M >0$ are natural numbers, the finite sums
$$ \sum_{i=1}^N n_i X_i \qquad \qquad \mathrm{and} \qquad \qquad  \sum_{j=1}^M m_j f_j $$ 
are well defined.

An \textit{additive functor} $(F,\sum):\cP_1 \rightarrow \cP_2 $
between two Picard stacks is a morphism of stacks $F: \cP_1 
\rightarrow \cP_2$ endowed with a natural isomorphism $\sum_{X,Y}: F(X +_{\cP_1} Y) \cong 
F(X) +_{\cP_2} F(Y)$ (for all $X,Y \in \cP_1(U)$) which is compatible with the natural 
isomorphisms $\ass$ and $\comm$ underlying 
$\cP_1$ and $\cP_2$.

 A \textit{morphism of additive functors $u:(F,\sum) \Rightarrow (F',\sum') $} is a morphism of cartesian functors which is compatible with the natural isomorphisms $\sum$ and $\sum'$ of $F$ and $F'$ respectively.
The stack $\cA\mathrm{rr}(\cP,\cQ),$ whose objects are additive functors from $\cP$ to $\cQ$ and whose arrows are morphisms of additive functors, is a stack in groupoid: any morphism of additive functors is invertible (i.e. it is an isomorphism of additive functors).

\medskip

\begin{definition}
	The \emph{product of two stacks $\cP$ and $\cQ$ over $\bS$} is the stack 
	$\cP \times_\bS \cQ$ over $\bS$ defined as followed:
	\begin{itemize}
		\item an object of the category $\cP \times_\bS \cQ(U)$ is a pair $(X,Y)$ of objects with $X$ an object of $\cP(U)$ and $Y$ an object of $\cQ(U)$;
		\item if $(X,Y)$ and $(X',Y')$ are two objects of 
		$\cP \times_\bS \cQ(U)$, an arrow of $\cP \times \cQ(U)$ from $(X,Y)$ to $(X',Y')$ is a pair $(f,g)$ of arrows with $f:X \rightarrow  X' $ an arrow of $\cP(U)$ and $g:Y \rightarrow  Y' $ an arrow of $\cQ(U)$;
	\end{itemize}
\end{definition}

For the notion of product of two Picard stacks see \cite[Definition 2.1]{B11} or \cite[Definition 2.1]{BeTa}. If $\cP$ is a stack, for any natural number $n>0$, $\cP^n $ denotes the product of $\cP$ with itself $n$ times, that is $\cP^n:= \cP \times_\bS \cdots \times_\bS  \cP$. Clearly $ \cP^n \times_\bS  \cP^m = \cP^{n+m}.$
We view an object of $\cP^n$ as a column vector $(X_1, \dots , X_n)^t$ of length $n$ (here $t$ is the transpose) and an arrow of $\cP^n$ as a row vector $(f_1, \dots , f_n)$ of length $n$. 

\medskip

A \textit{2-stack} over the site $\bS$ is a fibered 2-category $\twoX$ over $\bS$ such that
\begin{itemize}
	\item  2-descent is effective for objects in $\twoX$, and
	\item  for any $U$ and for every pair of objects $X,Y$ of the 2-category $\twoX(U)$, the fibered category of arrows $\cA\mathrm{rr}_{\twoX(U)}(X,Y)$ of $\twoX(U)$ is a stack over $\bS_{|U}.$
\end{itemize}

For the notions of morphism of 2-stack  (i.e. cartesian 2-functor), morphism of cartesian 2-functors, and modification of morphisms of cartesian 2-functors, we refer to \cite[Chp I]{Hakim}. A \textit{2-equivalence of 2-stacks} $F: \twoX_1 \to \twoX_2$ is a morphism of 2-stacks such that any object $Y $ of $\twoX_2(U)$ is equivalent to an object of the form $F(U)(X)$ with $X$ an object of $\twoX_1(U)$, and for any pair of objects $X,Y  $ of $\twoX_1(U)$, the morphism of stacks 
$\cA\mathrm{rr}_{\twoX_1(U)}(X,Y) \to \cA\mathrm{rr}_{\twoX_2(U)}(F(U)(X),F(U)(Y))$ is an equivalence of stacks.
Two 2-stacks are \textit{2-equivalent} if there exists a 2-equivalence of 2-stacks between them.
 A \textit{2-stack in 2-groupoids} over $\bS$ is a 2-stack $\twoX$ over $\bS$ such that the 2-category $\twoX(U)$ is a 2-groupoid, i.e. a 2-category whose 1-arrows are invertible up to a 2-arrow and whose 2-arrows are strictly invertible.  

\medskip

 Stacks over $\bS$ form a 2-stack over $\bS$, denoted 
$$ \STACK$$
 whose objects are stacks and whose hom-stack consists of morphisms of stacks and morphisms of cartesian functors. Picard stacks over $\bS$ form a 2-stack in 2-groupoids over $\bS$, denoted 
$$\PICARD$$
 whose objects are Picard stacks and whose hom-stack in groupoid consists of additive functors and morphisms of additive functors.

\medskip

We generalize from categories to 2-stacks \cite[Chapter 1, Definition 1.1]{AdRoVi}.

\begin{definition}\label{def-alg-2-st-2-alg}
	 An \textit{algebraic 2-stack theory} over $\bS$ is a 2-stack $\twoT$ over $\bS$ with finite products. \\
		A \textit{2-algebra} for the algebraic 2-stack theory $\twoT$ is a morphism of 2-stacks 
	$$A: \twoT \longrightarrow \STACK$$
	 preserving finite products.
	\end{definition}

	  Denote by $\Alg \twoT $
	  the 2-stack of 2-algebras for $\twoT$, whose hom-stack consists of morphisms of cartesian 2-functors and modification of morphisms of cartesian 2-functors.

	  \begin{definition}\label{alg-2-st}
		 A 2-stack $\twoX$ over $\bS$ is \textit{algebraic} if it is 2-equivalent to $\Alg \twoT$ for some algebraic 2-stack theory $\twoT$ over $\bS$.
\end{definition}

\section{Algebraic Picard stacks}\label{Alg-Pic-St}

Let $\bfu$ be the Picard stack whose only object is the unit object $ \mathbf{e}$ and whose only 1-arrow is the identity $\id_\mathbf{e}$. According to the dictionary between Picard stacks over $\bS$ and length one complexes of abelian sheaves on $\bS$ proved in~\cite[\S 1.4]{De},
the length one complex associated to the Picard stack $\bfu$ is $[  \mathbf{E} \stackrel{\id_\mathbf{E}}{\rightarrow} \mathbf{E}],$ where $\mathbf{E}$ the final object of the category of abelian sheaves on $\bS$.

\begin{definition}
 Denote by $\twoT_{\bfu}$ the 2-stack in which 
\begin{itemize}
	\item the objects are the Picard stacks $\bfu, \bfu^2, \bfu^3,\dots$. For any natural number $n>0,$ the Picard stack $\bfu^n$ admits only one object, the $n$-uplet $(\mathbf{e}, \dots, \mathbf{e})^t$, and only one arrow, the $n$-uplet $(\id_\mathbf{e}, \dots, \id_\mathbf{e})$.
	\item a 1-arrow from $\bfu^n$ to $\bfu^k$
	$$\underline{P}: \bfu^n \to \bfu^k$$
	  is a morphism of stacks defined by a matrix $P=(p_{ij})_{i=1, \dots,k \atop j=1, \dots ,n}$ of integers with $n$ columns and $k$ rows in the following way:
	  \begin{itemize}
	  	\item the object function: $\underline{P}$ sends the only object of $\bfu^n$, the $n$-uplet $(\mathbf{e}, \dots, \mathbf{e})^t$, to the only object of  $\bfu^k$, the $k$-uplet $(\mathbf{e}, \dots, \mathbf{e})^t$:
	  	$$\underline{P} \left(\begin{matrix} \mathbf{e} \\ \vdots \\ \mathbf{e}\end{matrix} \right) := P \left(\begin{matrix} \mathbf{e} \\ \vdots \\ \mathbf{e}\end{matrix}\right)
	  	=  \left(\begin{matrix} \sum_{j=1}^np_{1j}\mathbf{e} \\ \vdots \\ \sum_{j=1}^np_{kj}\mathbf{e}\end{matrix}\right)
	  	=  \left(\begin{matrix} \mathbf{e} \\ \vdots \\ \mathbf{e}\end{matrix}\right).$$
	  	\item  the arrow function: $\underline{P}$ sends the only arrow of $\bfu^n$, the $n$-uplet of arrows $(\id_\mathbf{e}, \dots, \id_\mathbf{e})$, to the only arrow of $\bfu^k$, the $k$-uplet of arrows $(\id_\mathbf{e}, \dots, \id_\mathbf{e})$. 
	  \end{itemize}
	The composition of two functors $\underline{P}: \bfu^n \to \bfu^k$ and $\underline{Q}: \bfu^k \to \bfu^m$ is the functor $\underline{QP}: \bfu^n \to \bfu^m$, where 
	$QP$ is the product of the two matrices $Q$ and $P$.
	\item a 2-arrow $u: \underline{P} \Rightarrow \underline{Q}$ between two morphisms of stacks $\underline {Q}, \underline{P}: \bfu^n \to \bfu^k$ is a morphism of cartesian functors which assigns to the only object $(\mathbf{e}, \dots, \mathbf{e})^t$ of $\bfu^n$ the only arrow  $(\id_\mathbf{e}, \dots, \id_\mathbf{e})$ of $\bfu^k$:
	\begin{equation}\label{arrow1^k}
\qquad		u_{(\mathbf{e}, \dots, \mathbf{e})^t} = ( \id_\mathbf{e} \dots \id_\mathbf{e}) :  \underline{P} \left(\begin{matrix} \mathbf{e} \\ \vdots \\ \mathbf{e}\end{matrix}\right) \longrightarrow   \underline{Q} \left(\begin{matrix} \mathbf{e} \\ \vdots \\ \mathbf{e}\end{matrix}\right).
	\end{equation}	
\end{itemize}
\end{definition}

 The 2-stack $\twoT_{\bfu}$ over $\bS$ is an algebraic 2-stack theory over $\bS$.

\begin{theorem}\label{Teo}
	The 2-stack $\PICARD$ is 2-equivalent to the 2-stack $\Alg \twoT_{\bfu} .$ 
\end{theorem}

\begin{proof}
	We construct a 2-equivalence of 2-stacks
	$$ \widehat{(\; \;)}:\PICARD \longrightarrow \Alg \twoT_{\bfu} $$
from the 2-stack of Picard stacks to the 2-stack of 2-algebras for the algebraic 2-stack theory $\twoT_{\bfu}.$ We proceed in several steps.

(1) To any Picard stack $\cP$ we associate a 2-algebra for $\twoT_{\bfu}$, that is a morphism of 2-stacks 
$$\widehat{\cP}: \twoT_\bfu \longrightarrow \STACK,$$
 preserving finite products, in the following way:
 \begin{itemize}
 	\item the object function: for any object $\bfu^n$ of $\twoT_\bfu$, we set $\widehat{\cP}(\bfu^n):=\cP^n. $ Since 
 	$$\widehat{\cP}(\bfu \times_\bS \bfu) =\widehat{\cP}(\bfu^2)=\cP^2 = \cP \times_\bS \cP = \widehat{\cP}(\bfu) \times_\bS \widehat{\cP}(\bfu),$$
 	 the functor  $\widehat{\cP}$ preserves finite products.
 	 
 	\item the 1-arrow function: for any 1-arrow
 	$\underline{P}: \bfu^n \to \bfu^k$ of $\twoT_\bfu$, the morphism of stacks
 	$$\widehat{\cP}(\underline{P}): \cP^n \longrightarrow  \cP^k $$
 	sends 
 	\begin{itemize}
 		\item the $n$-uplet $(X_1, \dots, X_n)^t$ of $\cP^n(U)$ to the $k$-uplet 
 		$$\widehat{\cP}(\underline{P}) \left(\begin{matrix} X_1 \\ \vdots \\ X_n\end{matrix}\right) :=P \left(\begin{matrix} X_1 \\ \vdots \\ X_n\end{matrix}\right)
 		=  \left(\begin{matrix} \sum_{j=1}^np_{1j} X_j \\ \vdots \\ \sum_{j=1}^np_{kj} X_j\end{matrix}\right)$$
 		of  $\cP^k(U).$
 		\item the arrow $(f_1, \dots, f_n): (X_1, \dots, X_n)^t \to (Y_1, \dots, Y_n)^t$ of $\cP^n(U)$ to the arrow 
 		$$\widehat{\cP}(\underline{P}) (f_1, \dots, f_n):= \widehat{\cP}(\underline{P})  \left(\begin{matrix} X_1 \\ \vdots \\ X_n\end{matrix}\right) \longrightarrow P  \left((f_1, \dots, f_n) \left(\begin{matrix} X_1 \\ \vdots \\ X_n\end{matrix}\right) \right) = \widehat{\cP}(\underline{P}) \left(\begin{matrix} Y_1 \\ \vdots \\ Y_n\end{matrix}\right) $$
 		of  $\cP^k(U).$ Clearly if $(\id_\mathbf{e}, \dots, \id_\mathbf{e})$ is the identity arrow of $\cP^n(U)$, then  $\widehat{\cP}(\underline{P}) (\id_\mathbf{e}, \dots, \id_\mathbf{e})$ is the identity arrow of $\cP^k(U)$, and if
 		 $(g_1, \dots, g_n): (Y_1, \dots, Y_n)^t \to (Z_1, \dots, Z_n)^t$ is another arrow of $\cP^n(U),$ we have that 
 		 \[\widehat{\cP}(\underline{P}) (g_1, \dots, g_n) \circ \widehat{\cP}(\underline{P}) (f_1, \dots, f_n) = \widehat{\cP}(\underline{P}) \big ( (g_1, \dots, g_n) \circ (f_1, \dots, f_n) \big). \]
 	\end{itemize}
 	\item the 2-arrow function: for any 2-arrow $u: \underline{P} \Rightarrow \underline{Q}$ of $\twoT_\bfu$ between two morphisms of stacks $\underline {Q}, \underline{P}: \bfu^n \to \bfu^k,$ the morphism of cartesian functors 
 		$$\widehat{\cP}(u):\widehat{\cP}(\underline{P}) \Longrightarrow \widehat{\cP}(\underline{Q} )$$
 		assigns to an $n$-uplet $(X_1, \dots, X_n)^t$ of $\cP^n(U)$ the arrow of $\cP^k(U)$ which is the image of the arrow $u_{(\mathbf{e}, \dots, \mathbf{e})^t}$ \eqref{arrow1^k} of $\bfu^k$ via the identification $\widehat{\cP}(\bfu^k)= \cP^k$:
 		\[\widehat{\cP}(u)_{(X_1, \dots, X_n)} := \widehat{\cP}(u_{(\mathbf{e}, \dots, \mathbf{e})^t})_{(X_1, \dots, X_n)}
 		 : \widehat{\cP}(\underline{P}) \left(\begin{matrix} X_1 \\ \vdots \\ X_n\end{matrix}\right) \longrightarrow \widehat{\cP}(\underline{Q}) \left(\begin{matrix} X_1 \\ \vdots \\ X_n\end{matrix}\right).  \]
 		For any arrow $(f_1, \dots, f_n): (X_1, \dots, X_n)^t \to (Y_1, \dots, Y_n)^t$ of $\cP^n(U)$ we have the following commutative diagram in $\cP^k(U)$
 	\begin{equation*}
 	\begin{tabular}{c}
 		\xymatrix{  \widehat{\cP}(\underline{P})(X_1, \dots, X_n)^t \ar[r]^{\widehat{\cP}(u)_{(X_1, \dots, X_n)}} \ar[d]_{\widehat{\cP}(\underline{P}) (f_1, \dots, f_n)} \qquad &  \qquad \widehat{\cP}(\underline{Q})(X_1, \dots, X_n)^t \ar[d]^{\widehat{\cP}(\underline{Q}) (f_1, \dots, f_n)} \\ 
 			\widehat{\cP}(\underline{P})(Y_1, \dots, Y_n)^t \ar[r]_{\widehat{\cP}(u)_{(Y_1, \dots, Y_n)}} \qquad & \qquad   \widehat{\cP}(\underline{Q})(Y_1, \dots, Y_n)^t.}
 	\end{tabular}
 \end{equation*}
 \end{itemize}

	(2) Any additive functor $F: \cP_1 \to \cP_2$ defines a morphism of cartesian 2-functors 
	$\widehat{F}: \widehat{\cP_1} \Rightarrow  \widehat{\cP_2}$. In fact,
	for any object $\bfu^n$  of $\twoT_\bfu$, we set 
	\begin{equation}\label{eq:F^1}
 \widehat{F}_{\bfu^n}:=F^n: \widehat{\cP_1}(\bfu^n)= {\cP_1}^n  \stackrel{F^n}{\longrightarrow} {\cP_2}^n =\widehat{\cP_2}(\bfu^n),
	\end{equation}
	and for any arrow $\underline{P}:\bfu^n \to \bfu^k$ of $\twoT_\bfu$ we have the following commutative diagram in $\STACK$ which involves stacks and  morphisms of stacks
\begin{equation*}
	\begin{tabular}{c}
		\xymatrix{ {\cP_1}^n \ar[r]^{\widehat{F}_{\bfu^n}} \ar[d]_{\widehat{\cP_1}(\underline{P})} & {\cP_2}^n \ar[d]^{\widehat{\cP_2}(\underline{P})}\\  {\cP_1}^k \ar[r]_{\widehat{F}_{\bfu^k}} & {\cP_2}^k.}
	\end{tabular}
\end{equation*}

(3)	 Any morphism $u: F \Rightarrow G$ between two additive functors $F,G: \cP_1 \to \cP_2$ defines a modification of morphisms of cartesian 2-functors 
	$ \widehat{u}: \widehat{F} \Rrightarrow \widehat{G}$.
	Indeed, for any object $\bfu^n$ of $\twoT_\bfu$, we set 
	 $$\widehat{u}_{\bfu^n}:=u^n= (\widehat{F}_{\bfu^n}: {\cP_1}^n  \stackrel{F^n}{\rightarrow} {\cP_2}^n) \Longrightarrow  (\widehat{G}_{\bfu^n}: {\cP_1}^n  \stackrel{G^n}{\rightarrow} {\cP_2}^n) .$$

(4) Now we check that the morphism of 2-stacks
$ \widehat{(\; \;)}:\PICARD \longrightarrow \Alg \twoT_{\bfu} $ is a 2-equivalence of 2-stacks. We start showing that 
if $A: \twoT_{\bfu} \to \STACK$ is a 2-algebra for the algebraic $2$-stack theory $\twoT_{\bfu},$ then $A$ is isomorphic (as morphism of 2-stacks) to $\widehat{\cP}$ for some Picard stack $\cP$.
Denote by $\cP$ the stack $A(\bfu)$. Consider the matrix $(1,1)$, which defines a morphism of stacks $\underline{(1,1)}: \bfu^2 \to \bfu$ in $\twoT_{\bfu}.$ The image 
of this 1-arrow via the morphism of 2-stacks $A$ 
\begin{align*}
	 A(\underline{(1,1)}) : \cP^2 &\longrightarrow \cP\\
	 \left( \begin{matrix}
	 	X_1\\ X_2
	 \end{matrix} \right)& \longmapsto (1,1)  \left( \begin{matrix}
 X_1\\ X_2
\end{matrix} \right) =X_1+X_2
\end{align*}
is a morphism of stacks, that defines a group law $+: \cP^2 \to \cP$ on the stack $\cP$.
The image of the 2-arrow of $\twoT_{\bfu}$
\[ \underline{(1,1)} \circ (\id_{\bfu},\underline{(1,1)}) \Rightarrow \underline{(1,1)} \circ (\underline{(1,1)},\id_{\bfu}) \]
 via the 2-algebra $A$ is the natural isomorphism 
\[ +\circ (\id_\cP \times +) \Rightarrow  +\circ (+\times \id_\cP)  \]
 which expresses the associativity constraint of the group law $+$ of $\cP$.
In an analogous way, the image of the 2-arrow of $\twoT_{\bfu}$
\[ \underline{(1,1)} \circ \underline{\left( \begin{matrix}
0&1\\1&0
\end{matrix} \right)} \Rightarrow \underline{(1,1)} \]
 via $A$ is the natural isomorphism
\[ +\circ \mathsf{s} \Rightarrow +,  \]
where $ \mathsf{s}(X,Y)=(Y,X)$ for all $X,Y \in \cP(U)$, which expresses the commutativity constraint of the group law $+$ of $\cP$.
Hence the stack $\cP$ is a Picard stack. 

Observe that for any $n>0$, 
$$A(\bfu^n)= A(\bfu \times_\bS \dots \times_\bS \bfu )= A(\bfu) \times_\bS \dots \times_\bS A(\bfu )
= \cP \times_\bS \dots \times_\bS \cP = \cP^n = \widehat{\cP}(\bfu^n),$$
that is the two 2-algebras $A$ and $ \widehat{\cP}$ take the same values on objects.
 Moreover for any 1-arrow
$\underline{P}: \bfu^n \to \bfu^k$ of $\twoT_\bfu$, the two morphisms of stacks
$$\widehat{\cP}(\underline{P}), A(\underline{P}): \cP^n \longrightarrow  \cP^k $$ coincide on objects $(X_1, \dots, X_n)^t$ and on arrows  $(f_1, \dots, f_n): (X_1, \dots, X_n)^t \to (Y_1, \dots, Y_n)^t$ of $\cP^n(U)$:
\[ A(\underline{P}) (X_1, \dots, X_n)^t =P (X_1, \dots, X_n)^t = \widehat{\cP}(\underline{P})  (X_1, \dots, X_n)^t,  \]
\[ A(\underline{P}) (f_1, \dots, f_n) =\widehat{\cP}(\underline{P})  (f_1, \dots, f_n) : P (X_1, \dots, X_n)^t \to   P ( (f_1, \dots, f_n) (X_1, \dots, X_n)^t) .  \]
 Finally for any 2-arrow $u: \underline{P} \Rightarrow \underline{Q}$ of $\twoT_\bfu$ between two morphisms of stacks $\underline {Q}, \underline{P}: \bfu^n \to \bfu^k,$ the two morphisms of cartesian functors 
$$\widehat{\cP}(u),A(u): \widehat{\cP}(\underline{P}) \Longrightarrow \widehat{\cP}(\underline{Q} )$$ also coincide. Hence
 the 2-algebras $A$ and $\widehat{\cP} $ are isomorphic as morphism of 2-stacks.

If $\cP_1$ and $\cP_2$ are two Picard stacks,  the morphism of stacks 
\begin{align*}
\cA\mathrm{rr}_{\PICARD}(\cP_1,\cP_2) &\longrightarrow  \cA\mathrm{rr}_{\Alg \twoT_{\bfu}}(\widehat{\cP_1},\widehat{\cP_2}), \\
F& \longmapsto \widehat{F}
\end{align*}
 is an equivalence of stacks,
since by \eqref{eq:F^1}  $F =  \widehat{F}_{\bfu}.$
We can conclude that $ \widehat{(\; \;)}:\PICARD \rightarrow \Alg \twoT_{\bfu} $ is a 2-equivalence of 2-stacks.
\end{proof}

\begin{corollary}
	The 2-stack $\PICARD$ of Picard stacks over $\bS$ is algebraic.
\end{corollary}

 Via the dictionary between Picard stacks over $\bS$ and length one complexes of abelian sheaves on $\bS$ given in \cite[\S 1.4]{De},
 1-motives can be seen as Picard stacks. This implies that 
 Picard stacks have an important role in the theory of motives (see \cite{B09,B09bis,B12,B19,BeTa2,BPSS,BP,BM09,BB,BG,BG2}).
We hope that this paper will shed some light on the group law of Picard stacks. 

\section{Acknowledgements}

We would like to thank the referees for careful reading.

\section{Consent for publication}

Not applicable.

 \section{Availability of data and materials}
 
 Not applicable.
 
 \section{Competing interests}
 
 Not applicable.
 
 \section{Funding}
 
 Federica Galluzzi was supported by the Italian Ministry of University and Research through the PRIN project n. 2022L34E7W.
 Cristiana Bertolin was supported by the Italian Ministry of University and Research through the PRIN project n. 20222B24AY.
 
 \section{Authors' contributions}
 All authors wrote the manuscript. All authors reviewed the manuscript. All authors approved the final version.
 
\bibliographystyle{plain}

\begin{thebibliography}{10}

\bibitem{AdRoVi} J. Adamek, J. Rosicky, E. M. Vitale.
\newblock Algebraic Theories. A categorical introduction to general algebra
\newblock Cambridge University Press, 2010.



\bibitem{B09bis} 
C. Bertolin.
Extensions and biextensions of locally constant group schemes, tori and abelian schemes.
\newblock Math. Z. 261 (2009), no. 4, pp. 845--868.


\bibitem{BM09} C. Bertolin, C. Mazza.
\newblock Biextensions of 1-motives in Voevodsky's category of motives.
\newblock Int. Math. Res. Not. IMRN 2009 (2009), no. 19, pp. 3747--3757.


\bibitem{B09} 
C. Bertolin.
\newblock Multilinear morphisms between 1-motives. 
\newblock J. Reine Angew. Math. 637 (2009), pp. 141--174.

\bibitem{B11} 
C. Bertolin.
\newblock Extensions of Picard stacks and their homological interpretation. \newblock J. Algebra 331 (2011), no. 1, pp.28--45.


\bibitem{B12} 
C. Bertolin.
\newblock Homological interpretation of extensions and biextensions of 1-motives.
\newblock J. Number Theory 132 (2012), no. 10, pp.2103--2131.

\bibitem{B13} 
C. Bertolin.
\newblock Biextensions of Picard stacks and their homological interpretation. 
\newblock Adv. Math. 233 (2013), no. 1, pp. 1--39.


\bibitem{BeTa}
C. Bertolin, A. Tatar.
\newblock Extensions of Picard 2-stacks and the cohomology groups $\mathrm{Ext}^i$ of length 3 complexes.
\newblock Ann. Mat. Pura Appl. 193 (2014), no. 1, pp. 291--315.


\bibitem{BeTa2}
C. Bertolin, A. Tatar.
\newblock Higher dimensional study of extensions via torsors.                                                                                                            
\newblock Ann. Mat. Pura Appl. 197 (2018), no. 2, pp. 433--468.

\bibitem{BB}
C. Bertolin, S. Brochard.
\newblock Morphisms of 1-motives defined by line bundles.  
\newblock Int. Math. Res. Not. IMRN 2019 (2019), no.5, pp. 1568--1600.

\bibitem{B19} 
C. Bertolin.
\newblock Third kind elliptic integrals and 1-motives. With a letter of Y. Andr\'e and an appendix by M. Waldschmidt.
\newblock J. Pure Appl. Algebra 224, no.10 (2020), 106396.

\bibitem{BG} 
C. Bertolin, F. Galluzzi.
\newblock  A note on divisorial correspondences of  extensions of abelian schemes by tori.
\newblock Comm. Algebra 48 (2020), no.7, pp. 3031--3034.


\bibitem{BG2} 
C. Bertolin, F. Galluzzi.
\newblock Brauer groups of 1-motives.                                           											 	    
\newblock J. Pure Appl. Algebra 225 (2021), no.11, 106754 , 22 pp.  

\bibitem{BPSS}
C. Bertolin, P. Philippon,  B. Saha, E Saha.
\newblock Semi-abelian analogues of Schanuel Conjecture and applications.
\newblock J. Algebra 596 (2022), pp. 250--288.          


\bibitem{BP}
C. Bertolin, P. Philippon.
\newblock Motivic Weil pairing and motivic Galois group.
\newblock submitted.






\bibitem{De} P. Deligne.
\newblock La formule de dualité globale
\newblock SGA 4 III, Exposé XVIII, 1973.


\bibitem{Gi} J. Giraud.
\newblock Cohomologie Non Abelienne
\newblock Springer Verlag, 1971.

\bibitem{Hakim}
M. Hakim.
\newblock Topos annel\'es et sch\'emas relatifs.
\newblock Ergebnisse der Mathematik und ihrer Grenzgebiete, Band 64. Springer-Verlag, Berlin-New York, 1972.


\bibitem{La} F. W. Lawvere.
\newblock Functorial semantics of algebraic theories.
\newblock Dissertation, Columbia University, 1963.




\end{thebibliography}

\end{document}